\documentclass[a4paper,12pt]{article}
\usepackage{amsfonts,amssymb,amsmath,amsthm,bm}
\usepackage{latexsym}
\usepackage{graphicx,color}
\usepackage{pst-solides3d}\usepackage{tikz}\usepackage{pst-3dplot}

\newcommand{\Section}[1]{
   \stepcounter{section}
   \bigskip\noindent
   {\large\bf\hbox{\thesection~~}#1}\par
   \nopagebreak
   \medskip
   \renewcommand{\theequation}{\thesection.\arabic{equation}}
   \setcounter{equation}{0}
   \setcounter{subsection}{0}
}

\newtheorem{thm}{Theorem}[section]
\newtheorem{cor}{Corollary}[section]
\newtheorem{lem}{Lemma}[section]
\newtheorem{rem}{Remark}[section]

\newcommand{\be}{\begin{equation}}\newcommand{\ee}{\end{equation}}
\newcommand{\bes}{\begin{equation*}}\newcommand{\ees}{\end{equation*}}
\newcommand{\ba}{\begin{array}}\newcommand{\ea}{\end{array}}
\newcommand{\ben}{\begin{eqnarray}}\newcommand{\een}{\end{eqnarray}}
\newcommand{\bn}{\begin{eqnarray*}}\newcommand{\en}{\end{eqnarray*}}

\begin{document}
\title{\Large \bf Counting Polynomials with Distinct Zeros in Finite Fields
\thanks{Research  is supported
  in part by 973 Program (2013CB834203),  National Natural Science Foundation of China under Grant No.61202437 and 11471162, in part by Natural Science Basic Research Plan in Shaanxi Province of China under Grant No.2015JM1022 and Natural Science Foundation of the Jiangsu Higher Education Institutes of China under Grant No.13KJB110016.}}
\author{\small Haiyan Zhou$^a$\ \ Li-Ping Wang$^b$\ \ Weiqiong Wang$^{c\dag}$\\ \small $a.$ School of Mathematics, Nanjing Normal University, Nanjing 210023, China\\
\small Email: haiyanxiaodong@gmail.com\\
\small $b.$ Institute of Information Engineering, Chinese Academy of Sciences
Beijing 100093, China\\
\small Email: wangliping@iie.ac.cn\\
\small $c.$ School of Science,
Chan'an University, Xi'an 710064, China \\ \small Email:wqwang@chd.edu.cn}
\date{}
 \maketitle

\noindent\rule[2mm]{\textwidth}{.01pt}
 {\footnotesize \noindent{\bf Abstract} Let $\mathbb{F}_q$ be a finite field with $q=p^e$ elements, where $p$ is a prime and $e\geq 1$ is an integer. Let $\ell<n$ be two positive integers.  
Fix a monic polynomial $u(x)=x^n +u_{n-1}x^{n-1}+\cdots +u_{\ell+1}x^{\ell+1} \in \mathbb{F}_q[x]$ of degree $n$ and 
consider all degree $n$ monic polynomials of the form 
$$f(x) = u(x) + v_\ell(x), \ v_\ell(x)=a_\ell x^\ell+a_{\ell-1}x^{\ell-1}+\cdots+a_1x+a_0\in \mathbb{F}_q[x].$$
For integer $0\leq k \leq {\rm min}\{n,q\}$,  
let $N_k(u(x),\ell)$ denote the total number of $v_\ell(x)$ such that $u(x)+v_\ell(x)$ has exactly $k$ distinct roots in $\mathbb{F}_q$, i.e. 
$$N_k(u(x),\ell)=|\{f(x)=u(x)+v_l(x)\ |\ f(x)\ {\rm has\ exactly}\ k\ {\rm distinct\ zeros\ in}\ \mathbb{F}_q\}|.$$ In this paper, we obtain explicit combinatorial formulae for $N_k(u(x),\ell)$ when $n-\ell$ is small, namely when $n-\ell= 1, 2, 3$. As an application, we define 
two kinds of Wenger graphs called jumped Wenger graphs and obtain their explicit spectrum. 

  }

   \noindent\rule[2mm]{\textwidth}{.1pt}

\Section{Introduction}

Let $\mathbb{F}_q$ be a finite field with $q=p^e$ elements, where $p$ is a prime and $e\geq 1$ is an integer. Let $\ell<n$ be two positive integers.  
Fix a monic polynomial $u(x)=x^n +u_{n-1}x^{n-1}+\cdots +u_{\ell+1}x^{\ell+1} \in \mathbb{F}_q[x]$ of degree $n$ and 
consider all degree $n$ monic polynomials of the form 
$$f(x) = u(x) + v_\ell(x), \ v_\ell(x)=a_\ell x^\ell+a_{\ell-1}x^{\ell-1}+\cdots+a_1x+a_0\in \mathbb{F}_q[x].$$
We are interesting in the number of distinct roots in $\mathbb{F}_q$ of $f(x)$ as the lower degree part $v_\ell(x)$ varies. 
Since $a^q = a$ for all $a \in \mathbb{F}_q$, we can reduce the polynomial $u(x)$ modulo $x^q-x$. In this way, we can and will assume that $n<q$. 

It is clear that $f(x)$ has at most $n$ distinct zeros in $\mathbb{F}_q$. For integer $0\leq k \leq n$,  
let $N_k(u(x),\ell)$ denote the total number of $v_\ell(x)$ such that $u(x)+v_\ell(x)$ has exactly $k$ distinct roots in $\mathbb{F}_q$, i.e. 
$$N_k(u(x),\ell)=|\{f(x)=u(x)+v_l(x)\ |\ f(x)\ {\rm has\ exactly}\ k\ {\rm distinct\ zeros\ in}\ \mathbb{F}_q\}|.$$
Understanding $N_k(u(x),\ell)$ is  an important number theoretical problem with a wide range of applications. 
For example, in coding theory, $v_\ell(x)$ represents a code word in the $[q, \ell+1]_q$ Reed-Solomon codes and $u(x)$ represents a received word. 
The number $N_k(u(x),\ell)$ is then  the number of code words whose distance to the received word $u(x)$ is precisely $q-k$. 
Determining the largest $k$ such that $N_k(u(x), \ell)>0$ is equivalent to computing the error distance from the received 
word $u(x)$ to the code, which is the most important problem in decoding Reed-Solomon codes, see \cite{KW}\cite{LiW}\cite{ZW} for various partial results.  
In the special case $k=n$ (the polynomial $f(x)$ splits as a product of $n$ distinct linear factors), 
the possible large size for $N_n(u(x),\ell)$ is the key to prove several complexity results in decoding primitive Reed-Solomon codes \cite{CW1} and in approximating the minimum distance 
of linear codes \cite{CW2}. 
For another example, in graph theory, $N_k(u(x), \ell)$ represents the multiplicity of 
certain eigenvalue in an important class of algebraic graphs extending the classical Wenger graph, see \cite{CLWWW}. Deciding this multiplicity is a difficult 
problem in general. 

Mathematically, the number $N_k(u(x),\ell)$ becomes increasingly more complicated as $n-\ell$ grows. Thus, we cannot expect an explicit formula for  $N_k(u(x),\ell)$ if $n-\ell$ is large. 
In the simplest case $n-\ell=1$,   A. Knopfmacher and J. Knopfmacher derived an explicit combinatorial formula for $N_k(u(x),\ell)$ in $\cite{AK90}$.   
Using this formula,  S.M. Cioab\u{a}, F. Lazebnik and W. Li obtained the explicit spectrum of Wenger graphs  in $\cite{CL14}$.  
In this paper, we obtain explicit combinatorial formulae for $N_k(u(x),\ell)$ when $n-\ell$ is small, namely when $n-\ell= 1, 2, 3$. In the case $n-\ell=1$,  we give a simple proof as a simple application of the inclusion-exclusion principal. In the case $n-\ell=2$, we use the inclusion-exclusion principe 
together with the subset sum result in \cite{LW1}. In the case $n-\ell=3$, it is more complicated. When $k=n$, this is an extension of the subset sum problem up to $2$ moments(called Moments Subset-Sum with parameter $2$). For a fixed $d\geq 1$, the Moments Subset-Sum with parameter $d$ is formally defined as follows, see \cite{GGG}:

\noindent{\bf Moments\  Subset-Sum( MSS(d)):} Given a set $A=\{a_1,\cdots,a_n\}$, $a_i\in\mathbb{F}_q$, integer $t$, and $m_1, \cdots, m_d\in\mathbb{F}_q$, decide if there exists a  subset $S\subseteq A$  of size $t$, satisfying  $\sum\limits_{a\in S}a^i=m_i$  for all $1\leq i\leq d$.

Note that MSS(1) is the usual subset sum problem and it is well-known for the NP-hardness of subset sum problem. However, it turns out to be much more difficult to prove NP-hardness for MSS(d) for $d\geq 2$. In 2015, V. Gandikota, B. Ghazi and E. Grigorescu proved the NP-hardness for MSS(d) for $d=2, 3$, see \cite{GGG}. Surprisingly, when the degree of the extension $\mathbb{F}_q/\mathbb{F}_p$ is even, $A=\mathbb{F}_q,\ t=n,\ m_i=0$ and $d=2$, we obtain an explicit combinatorial formula for the number of $S$, i.e., Theorem 4.1, employing the more advanced sieving formula 
from $\cite{LW2}$ together with results on quadratic equations over finite fields. Finally, we get explicit combinatorial formulae for $N_k(x^n, n-3)$ using the inclusion-exclusion principal 
together with Theorem 4.1. As an application, we define 
two kinds of Wenger graphs called jumped Wenger graphs and obtain their explicit spectrum. 
Note that in the cases $n-\ell =2, 3$, the smallest $k$ such that $N_k(u(x), \ell)>0$ is determined in \cite{LZ} by giving an explicit construction 
of a solution.  Our result gives an exact and explicit formula for $N_k(u(x), \ell)$. 

\Section{The case $n-\ell=1$}

In this simplest case,  using the generating function over an additive arithmetical semigroup,  A. Knopfmacher and J. Knopfmacher obtained an explicit combinatorial formula for $N_k(u(x),\ell)$ in $\cite{AK90}$.  Here, we would give the simple proof according to the  classical inclusion-exclusion principal.  We recall it briefly. 

Let $S$ be a finite set of objects and let $P_1,\ P_2,\ \cdots,\ P_m$ be $m$ properties referring to the objects in $S$. Let $I\subseteq\{P_1,\ \cdots,\ P_m\}$. Define $S_\emptyset=S$ and $S_I=\{x\in S|\ x\ {\rm satifies\ all\ properties\ in}\ I\}$ for $I\neq \emptyset$. For any non-negative integer $j$, we put $S_j=\{x\in S|\ x\ {\rm satifies\ exactly}\ j\ {\rm properts\ of}\ \{P_1,\ \cdots,\ P_m\}\}$. It is well-known that the  classical inclusion-exclusion principal implies 
$$|S_j|=\sum\limits_{|I|=j}|S_I|-\sum\limits_{|I|=j+1}|S_I|-\cdots+(-1)^{n-j}\sum\limits_{|I|=n}|S_I|,\ {\rm where }\ j=0.$$
It is worth mentioning that the above formula doesn't work for $j\geq 1$. For example, let $S$ be the set consisting of all monic polynomials over $\mathbb{F}_2$ with degree $3$, $P_1$ the property that the monic polynomial in $S$ has a zero $0$, and $P_2$ the property that the monic polynomial in $S$ has a zero $1$. Then it is easy to compute $|S_1|=4$ and $\sum\limits_{|I|=1}|S_I|-\sum\limits_{|I|=2}|S_I|=6$.

\begin{thm}\label{thm:n-l=1}

$$N_k(x^n,n-1)=q^{n-k}\left(\begin{array}{c}
q\\k
\end{array}\right)\sum\limits\limits_{i=0}^{n-k}(-1)^i\left(\begin{array}{c}
q-k\\i
\end{array}\right)q^{-i}.$$
\end{thm}

\begin{proof}Let $f(x)= x^n +a_1x^{n-1} +a_2x^{n-2}+\cdots +a_n$ and $c_1, c_2, \cdots, c_k$ be $k$ distinct roots of $f(x)$ in $\mathbb{F}_q$. Then there exists a polynomial 
$g(x)\in \mathbb{F}_q[x]$ such that 
$f(x)=(x-c_1)(x-c_2)\cdots(x-c_k)g(x)$.
For a fixed $k$-subset $J=\{c_1, \cdots, c_k\}\subset \mathbb{F}_q$ 
and a subset $I \subseteq \mathbb{F}_q -J$, we define the set 
$$S_I(J)=\left\{g(x)\in \mathbb{F}_q[x]|\ g(\alpha)=0 \ {\rm for } \ {\rm all} \ \alpha\in I \right\}.$$ 
For $0\leq |I| <n-k$, it is obvious to obtain that the cardinality of the set $S_I(J)$ is $q^{n-k-|I|}$. It follows that for $0\leq i<n-k$, 
$$\sum_{I\subseteq \mathbb{F}_q-J, |I|=i} |S_I(J)| = {q-k \choose i}q^{n-k-i}.$$ 
For $|I|>n-k$, it is clear that $|S_I(J)|=0$. 
By the inclusion-exclusion principle,  we deduce that 
$$\begin{array}{ll}N_k(x^n,n-1)&= \sum\limits_{|J|=k} \sum\limits_{I \subset \mathbb{F}_q-J} (-1)^{|I|} |S_I(J)|\\ 
&=\sum\limits_{i=0}^{n-k}(-1)^i  \sum\limits_{|J|=k} \sum\limits_{I \subset \mathbb{F}_q-J, |I|=i} |S_I(J)| \\
&=q^{n-k}\left(\begin{array}{c}
q\\k
\end{array}\right)\sum\limits\limits_{i=0}^{n-k}(-1)^i\left(\begin{array}{c}
q-k\\i
\end{array}\right)q^{-i}\end{array}.$$

\end{proof}

\begin{rem}\label{rem:remark1}
If the degree of $f(x)= x^n +a_1x^{n-1} +a_2x^{n-2}+\cdots +a_n$ is greater than $q-1$, i.e., $n\geq q$, then by the Euclidean division,  there exist $g(x), h(x)\in \mathbb{F}_q[x]$  such that $f(x)= (x^q-x)g(x)+h(x)$, where $g(x)$ is the monic polynomial with degree $n-q$ and the degree of $h(x)$ is less than $q$ or $h(x)=0$. If $f(x)$ has exactly $k\leq q-1$ distinct roots in $\mathbb{F}_q$, then  $h(x)$ has also exactly $k$ distinct roots in $\mathbb{F}_q$. So we obtain 
$$N_k(x^n,n-1)=q^{n-q}(q-1)\sum\limits_{r=k}^{q-1}q^{r-k}\left(\begin{array}{c}
q\\k
\end{array}\right)\sum\limits\limits_{i=0}^{r-k}(-1)^i\left(\begin{array}{c}
q-k\\i
\end{array}\right)q^{-i}.$$ 
Since $$\begin{array}{l}\sum\limits_{r=k}^{q-1}q^{r-k}\sum\limits\limits_{i=0}^{r-k}(-1)^i\left(\begin{array}{c}
q-k\\i
\end{array}\right)q^{-i}\\
\hskip 4cm =1+q\sum\limits_{i=0}^1(-1)^i{q-k \choose i}q^{-i}+\cdots+q^{q-1-k}\sum\limits_{i=0}^{q-1-k}(-1)^i{q-k \choose i}q^{-i}\\
\hskip 4cm =\frac{1}{1-q}\sum\limits_{i=0}^{q-k-1}(-1)^i{q-k \choose i}(1-q^{q-k-i})\\
\hskip 4cm =(q-1)^{q-k-1},
\end{array}$$ we have $$N_k(x^n,n-1)={q \choose k}q^{n-q}(q-1)^{q-k}.$$ 
If $k=q$, then $h(x)=0$. Then $N_k(x^n,n-1)=q^{n-q}$. Hence for $n\geq q$, $$N_k(x^n,n-1)={q \choose k}q^{n-q}(q-1)^{q-k}.$$ 
\end{rem}

\Section{The case $n-\ell=2$}

In the special case that $k=n$, this is the counting version for the $n$-subset sum problem over $\mathbb{F}_q$, 
which is already handled in \cite{LW1}. We state this result as a lemma and will use it in our proof. 

\begin{lem}(See $\cite{LW1}$)\label{lem:LW}
For $b\in \mathbb{F}_q$, let $M(n,b)$ be the number of $n$-subsets of $\mathbb{F}_q$ whose elements sum to b. If $p\nmid n$, then $$M(n,b)=\frac{1}{q}\left(\ba{c}q\\n\ea\right).$$ 
If $p|n$, then 
$$M(n,b)=\frac{1}{q}\left(\ba{c}q\\n\ea\right)+(-1)^{n+\frac{n}{p}}\frac{v(b)}{q}\left(\ba{c}q/p\\n/p\ea\right),$$ 
where $v(b)=-1$ if $b\neq 0$, and $v(b)=q-1$ if $b=0$.
\end{lem}

In terms of our earlier notations, we have $M(n, b) = N_n(x^n -bx^{n-1}, n-2)$.

\begin{thm}\label{thm:a1=0}
(i)\ If $p\nmid n$, then 

$$N_k(x^n-bx^{n-1},n-2)=q^{n-k-1}\left(\begin{array}{c}
q\\k
\end{array}\right)\sum\limits\limits_{i=0}^{n-k}(-1)^i\left(\begin{array}{c}
q-k\\i
\end{array}\right)q^{-i}.$$

(ii)\ If $p\mid n$, then  $$\begin{array}{rl}N_k(x^n -b x^{n-1},n-2)=&q^{n-k-1}\left(\begin{array}{c}
q\\k
\end{array}\right)\sum\limits\limits_{i=0}^{n-k}(-1)^i\left(\begin{array}{c}
q-k\\i
\end{array}\right)q^{-i}\\
&+(-1)^{\frac{n}{p}+n}\frac{v(b)}{q}\left(\begin{array}{c}
n\\k
\end{array}\right)
\left(\begin{array}{c}
q/p\\n/p
\end{array}\right)
\end{array}.$$

\end{thm}

\begin{proof} Let $c_1, c_2, \cdots, c_k$ be $k$ distinct roots of $f(x)$ in $\mathbb{F}_q$. Then there exists a polynomial 
$$g(x)=x^{n-k}+d_1x^{n-k-1}+\cdots+d_{n-k-1}x+d_{n-k}\in \mathbb{F}_q[x]$$
such that 
$$
\ba{ll}
f(x)&= x^n -bx^{n-1} +a_2x^{n-2}+\cdots +a_n \\
&=(x-c_1)(x-c_2)\cdots(x-c_k)g(x)\\
 &=(x^k+b_1x^{k-1}+\cdots+b_{k-1}x+b_k)(x^{n-k}+d_1x^{n-k-1}+\cdots+d_{n-k-1}x+d_{n-k}).
\ea $$
Comparing the coefficients, we have 
$$
\left\{
\ba{ccc}
b_1+d_1&=&-b\\
b_2+b_1d_1+d_2&=&a_2\\
b_3+b_2d_1+b_1d_2+d_3&=&a_3\\
\cdots&\cdots&\cdots\\
b_kd_{n-k}=a_n
\ea
\right.
$$
For fixed $b$ and fixed $k$-subset $J=\{c_1, \cdots, c_k\}\subset \mathbb{F}_q$,  the coefficient $d_1=-(b+c_1+c_2+\cdots+c_k)$ is then fixed. The other coefficients $\{ d_2, \cdots, d_{n-k}\}$ of the polynomial $g(x)$ 
are free since $\{ a_2,\cdots, a_n\}$ are free. 

For a subset $I \subseteq \mathbb{F}_q -J$, define the set 
$$\ba{ll}S_I(J)=&\left\{g(x)=x^{n-k}+d_1x^{n-k-1}+\cdots + d_{n-k} \in \mathbb{F}_q[x]|d_1=b-(c_1+c_2+\cdots+c_k), \right.\\ & \left.g(\alpha)=0 \ {\rm for } \ {\rm all} \ \alpha\in I \right\}.\ea$$ 
For $0\leq |I| <n-k$, the above argument shows that the cardinality of the set $S_I(J)$ is $q^{n-k-1-|I|}$. It follows that for $0\leq i<n-k$, 
$$\sum_{I\subseteq \mathbb{F}_q-J, |I|=i} |S_I(J)| = {q-k \choose i}q^{n-k-1-i}.$$ 
When $i=n-k$, $f(x)$ is forced to have $n$ distinct roots in $\mathbb{F}_q$ with sum equal to $b$. Then by Lemma \ref{lem:LW}, we deduce 

$$
\sum\limits_{|J|=k}\sum\limits_{I\subseteq \mathbb{F}_q-J, |I|=n-k} |S_I(J)|=\left\{\begin{array}{l}
\left(\begin{array}{c}
n\\k
\end{array}\right)\frac{1}{q}\left(\begin{array}{c}
q\\n
\end{array}\right), p\nmid n,\\
\left(\begin{array}{c}
n\\k
\end{array}\right)\left[\frac{1}{q}\left(\begin{array}{c}
q\\n
\end{array}\right)+(-1)^{n+\frac{n}{p}}\frac{v(b)}{q}\left(\begin{array}{c}
q/p\\n/p
\end{array}\right)\right], p\mid n, 
\end{array}
\right.
$$
For $|I|>n-k$, it is clear that $|S_I(J)|=0$. 
By the inclusion-exclusion principle,  we deduce that 
$$N_k(x^n-bx^{n-1},n-2)= \sum_{|J|=k} \sum_{I \subset \mathbb{F}_q-J} (-1)^{|I|} |S_I(J)|  
=\sum_{i=0}^{n-k}(-1)^i  \sum_{|J|=k} \sum_{I \subset \mathbb{F}_q-J, |I|=i} |S_I(J)| .$$
Hence,
$$
N_k(x^n -bx^{n-1},n-2)=\left\{\begin{array}{l}

\left(\begin{array}{c}
q\\k
\end{array}\right)\sum\limits\limits_{i=0}^{n-k-1}(-1)^i\left(\begin{array}{c}
q-k\\i
\end{array}\right)q^{n-k-1-i}+
\\(-1)^{n-k}\left(\begin{array}{c}
n\\k
\end{array}\right)\frac{1}{q}\left(\begin{array}{c}
q\\n
\end{array}\right), p\nmid n,\\

\left(\begin{array}{c}
q\\k
\end{array}\right)
\sum\limits\limits_{i=0}^{n-k-1}(-1)^i
\left(\begin{array}{c}
q-k\\i
\end{array}\right)q^{n-k-1-i}+
\\(-1)^{n-k}\left(\begin{array}{c}
n\\k
\end{array}\right)
\left[\frac{1}{q}\left(\begin{array}{c}
q\\n
\end{array}\right)+(-1)^{n+\frac{n}{p}}\frac{v(b)}{q}\left(\begin{array}{c}
q/p\\n/p
\end{array}\right)\right], p\mid n.
\end{array}
\right.
$$

This Theorem is proved from the fact $\left(\begin{array}{c}
q\\k
\end{array}\right)\left(\begin{array}{c}
q-k\\n-k
\end{array}\right)=\left(\begin{array}{c}
n\\k
\end{array}\right)\left(\begin{array}{c}
q\\n
\end{array}\right)$.

\end{proof}

\begin{rem}\label{rem:remark2}
For $n\geq q$, we can deduce the formula of $N_k(x^n-bx^{n-1}, n-2)$.

$1)$ If $n>q$, then $n-1>q-1$. Similar arguments to those used in the Remark $\ref{rem:remark1}$ show that 
$$N_k(x^n-bx^{n-1}, n-2)=q^{n-q-1}{q \choose k}(q-1)^{q-k}.$$

$2)$ If $n=q$, then $f(x)=x^q-x-bx^{q-1}+a_2x^{q-2}+\cdots+a_{q-2}x^2+(a_{q-1}+1)x+a_q$. It is easy to get the following conclusions:

When $b\neq 0$, $$
N_k(x^q-bx^{q-1}, q-2)=\left\{\begin{array}{ll}0& k=q,\\ \frac{1}{q} {q \choose k}((q-1)^{q-k}-(-1)^{q-k})& k\leq q-1.\end{array}\right.$$

When $b=0$, 
$$
N_k(x^q-bx^{q-1}, q-2)=\left\{\begin{array}{ll}1& k=q,\\ 0& k=q-1,\\ \frac{q-1}{q} {q \choose k}((q-1)^{q-k-1}+(-1)^{q-k})& k\leq q-2.\end{array}\right.$$
\end{rem}

\Section{The case $n-\ell=3$}

In this section, we always assume that $q$ is an odd number. Let $\mathbb{F}_q^n$ be the Cartesian product of $n$ copies of $\mathbb{F}_q$. For convenience, we firstly state some results on the number of common solutions in $\mathbb{F}_q^n$ of the equations

\be\label{def:equations}
\Bigg\{
\begin{array}{lll}
a_1x_1^2+\cdots+a_nx_n^2&=&a_0,\\
b_1x_1+\cdots+b_nx_n&=&b_0,
\end{array}
\ee
where $a_0, b_0, b_1,\cdots, b_n\in\mathbb{F}_q$, $a_1,\cdots, a_n\in\mathbb{F}_q^*$, $b_i\neq 0$ for at least one $i$, $1\leq i\leq n,$ see Exercises 6.31-6.34 in $\cite{Lind}$.

\begin{lem}\label{lem:equations}
Denote by $N(n, a_0,b_0)$ the number of common solutions in $\mathbb{F}_q^n$ of the equations $\eqref{def:equations}$. Put $a=a_1a_2\cdots a_n$, $b=b_1^2a_1^{-1}+\cdots+b_n^2a_n^{-1}$, $c=b_0^2-a_0b$. Then

$i)$ For $b\neq 0,\ c=0,$ $$
N(n, a_0, b_0)=\{\begin{array}{ll}
q^{n-2}& {\rm if}\ n\ {\rm even},\\
q^{n-2}+q^{(n-3)/2}(q-1)\chi((-1)^{(n-1)/2}ab)& {\rm if}\ n\ {\rm odd},
\end{array}$$
where $\chi$ is the quadratic character of $\mathbb{F}_q.$

$ii)$ For $b\neq 0,\ c\neq 0,$ $$
N(n, a_0, b_0)=\{\begin{array}{ll}
q^{n-2}+q^{(n-2)/2}\chi((-1)^{n/2}ac)& {\rm if}\ n\ {\rm even},\\
q^{n-2}-q^{(n-3)/2}\chi((-1)^{(n-1)/2}ab)& {\rm if}\ n\ {\rm odd}.
\end{array}$$

$iii)$ For $b=c=0,$ $$
N(n, a_0, b_0)=\{\begin{array}{ll}
q^{n-2}+v(a_0)q^{(n-2)/2}\chi((-1)^{n/2}a)& {\rm if}\ n\ {\rm even},\\
q^{n-2}+q^{(n-1)/2}\chi((-1)^{(n-1)/2}a_0a)& {\rm if}\ n\ {\rm odd},
\end{array}$$
where $v$ is as in Lemma \ref{lem:LW}.

$iv)$ For $b=0,\ c\neq 0$, $N(n, a_0, b_0)=q^{n-2}$.

\end{lem}

Now, we begin to recall a sieve for distinct coordinate counting, see $\cite{LW2}$. Let $X$ be a subset of $\mathbb{F}_q^n$. Motivated by diverse applications in coding theory and graph theory, it is very interesting to count the number of elements in the set
$$\overline{X}=\{(x_1, \cdots, x_n)\in X|x_i\neq x_j,\ \forall i\neq j\}.$$ In $\cite{LW2}$, J. Li and D. Wan discovered the new sieving formula about $|\overline{X}|$.
Let $S_n$ be the symmetric group on $\{1,2,\cdots,n\}$. For a given permutation $\tau=(i_1i_2\cdots i_{t_1})\cdots(l_1l_2\cdots l_{t_s})$ with $t_i\geq 1,\ 1\leq i\leq s$, define $$X_\tau=\{(x_1, \cdots, x_n)\in X|x_{i_1}=\cdots=x_{i_{t_1}},\ x_{l_1}=\cdots=x_{l_{t_s}}\}.$$ Now the symmetric group $S_n$ acts on $\mathbb{F}_q^n$ by permuting coordinates. That is, for given $\tau\in S_n$ and $x=(x_1, \cdots, x_n)\in \mathbb{F}_q^n$, we have $$\tau\circ x=(x_{\tau(1)}, \cdots, x_{\tau(n)})\in X.$$ $X$ is called symmetric if any $x\in X$ and any $\tau\in S_n$, $\tau\circ x\in X$. Furthermore, If $X$ satisfies the "strongly symmetric" condition, that is, for any $\tau$ and $\sigma$ in $S_n$, one has $|X_{\tau}|=|X_{\sigma}|$ provided $l(\tau)$ and $l(\sigma)$, then we call X a strongly symmetric set. Let $C_n$ be the set of conjugacy classes of $S_n$. If $X$ is symmetric, then 
$$|\overline{X}|=\sum\limits\limits_{\tau\in C_n}(-1)^{n-l(\tau)}C(\tau)|X_\tau|,$$
where $C(\tau)$ is the number of permutations conjugate to $\tau$ and $l(\tau)$ is the number of cycles including the trivial cycle.

 A permutation $\tau\in S_n$ is said to be of type $(c_1, c_2, \cdots, c_n)$ if $\tau$ has exactly $c_i$
 cycles of length $i$. We denote by $N(c_1, c_2, \cdots, c_n)$  the number of permutations in $S_k$ of type $(c_1, c_2, \cdots, c_n)$ and we have (see $\cite{S97}$),
 $$N(c_1, c_2, \cdots, c_n)=\frac{n!}{1^{c_1}{c_1}!2^{c_2}c_2!\cdots n^{c_n}c_n!}.$$
 Since two permutations are conjugate if and only if they have the same type, we have $C(\tau)=N(c_1, c_2, \cdots, c_n).$

\begin{lem}\label{lem:S}
Put
$$S_+(n)=\sum\limits_{{\sum\limits_{i=1}^nic_i=n},\ {\sum\limits_{p\nmid i}c_i}\ {\rm is\ even}}N(c_1, c_2, \cdots, c_n)\prod_{p|i}(-q)^{c_i}\prod_{p\nmid i}(-\sqrt{q})^{c_i},$$ 
$$S_{-}(n)=\sum\limits_{{\sum\limits_{i=1}^nic_i=n},\ {\sum\limits_{p\nmid i}c_i}\ {\rm is\ odd}}N(c_1, c_2, \cdots, c_n)\prod_{p|i}(-q)^{c_i}\prod_{p\nmid i}(-\sqrt{q})^{c_i}.$$
Then 

 $$S_+(n)=
\frac{n!}{2}((-1)^n\alpha(n)+\beta(n)),$$

$$S_-(n)=
\frac{n!}{2}((-1)^n\alpha(n)-\beta(n)),$$
where $$\alpha(n)=\sum\limits_{i+pj=n, 0\leq i\leq \sqrt{q}}\left(\begin{array}{c}\sqrt{q}\\i\end{array}\right)\left(\begin{array}{c}\frac{q-\sqrt{q}}{p}\\ j\end{array}\right)$$ and $$\beta(n)=\sum\limits_{i+pj=n, i\geq 0}(-1)^j\left(\begin{array}{c}\sqrt{q}-1+i\\ \sqrt{q}-1\end{array}\right)\left(\begin{array}{c}\frac{q+\sqrt{q}}{p}\\ j\end{array}\right).$$
\end{lem}

\begin{proof} 

Define the generating function$$C_n(t_1, t_2,\cdots, t_n)=\sum\limits_{\sum\limits_{i=1}^nic_i=n}N(c_1,c_2,\cdots, c_n)t_1^{c_1}t_2^{c_2}\cdots t_n^{c_n}.$$ Then we get the following exponential generating function$$\sum\limits_{n\geq 0}C_n(t_1, t_2,\cdots, t_n)\frac{u^n}{n!}=e^{ut_1+u^2\frac{t_2}{2}+u^3\frac{t_3}{3}+\cdots}.$$ For given generating function $f(x)$, denote by $[x^i]f(x)$ the coefficient of $x^i$ in the formal power series expansion of $f(x)$. 

$1)$ If $t_i=-\sqrt{q}$ for $p\nmid i$ and $t_i=-q$ for $p|i$, then we have

$\begin{array}{l}
C_n(-\sqrt{q},\cdots,-\sqrt{q},-q,-\sqrt{q},\cdots,-\sqrt{q},-q,\cdots)\\
\hskip 4cm=[\frac{u^n}{n!}]e^{-\sqrt{q}(u+\frac{u^2}{2}+\frac{u^3}{3}+\cdots)+\frac{-q+\sqrt{q}}{p}(u^p+\frac{u^{2p}}{2}+\frac{u^{3p}}{3}+\cdots)}\\
 \hskip 4cm=[\frac{u^n}{n!}]e^{\sqrt{q}\ln(1-u)+\frac{q-\sqrt{q}}{p}\ln(1-u^p)}\\
\hskip 4cm=[\frac{u^n}{n!}](1-u)^{\sqrt{q}}(1-u^p)^{\frac{q-\sqrt{q}}{p}}\\
\hskip 4cm=[\frac{u^n}{n!}](\sum\limits_{i\geq 0}(-1)^i\left(\begin{array}{c}\sqrt{q}\\i\end{array}\right)u^i)(\sum\limits_{j\geq 0}(-1)^j\left(\begin{array}{c}\frac{q-\sqrt{q}}{p}\\j\end{array}\right)u^{pj})\\
\hskip 4cm=n!\sum\limits_{i+pj=n,0\leq i\leq \sqrt{q}}(-1)^n\left(\begin{array}{c}\sqrt{q}\\ i\end{array}\right)\left(\begin{array}{c}\frac{q-\sqrt{q}}{p}\\ j\end{array}\right).
\end{array}$

Similarly, if $t_i=\sqrt{q}$ for $p\nmid i$ and $t_i=-q$ for $p|i$, then we have 
$$C_n(\sqrt{q},\cdots,\sqrt{q},-q,\sqrt{q}, \cdots)=n!\sum\limits_{i+pj=n, i\geq 0}(-1)^{j}\left(\begin{array}{c}\sqrt{q}-1+i\\ \sqrt{q}-1\end{array}\right)\left(\begin{array}{c}\frac{q+\sqrt{q}}{p}\\ j\end{array}\right).$$

Thus,  this lemma is proved from 
$$S_+(n)=\frac{C_n(-\sqrt{q},\cdots,-\sqrt{q},-q,-\sqrt{q},\cdots)+C_n(\sqrt{q},\cdots,\sqrt{q},-q,\sqrt{q},\cdots)}{2},$$
$$S_-(n)=\frac{C_n(-\sqrt{q},\cdots,-\sqrt{q},-q,-\sqrt{q},\cdots)-C_n(\sqrt{q},\cdots,\sqrt{q},-q,\sqrt{q},\cdots)}{2}.$$

\end{proof}

\begin{lem}\label{lem:CHAR}
Let $q=p^{r}$ with $p\neq 2$ and  $\chi$ be the quadratic character of $\mathbb{F}_q$. Then $\chi|_{F_p}$ is the trivial character of $F_p$ if and only if $r$ is even.
\end{lem}

\begin{proof}
Let $g$ be a primitive element of $F_p$. Then  $\chi|_{F_p}$ is the trivial character of $F_p$ if and only if $\chi(g)=1$, that is, $x^2-g=0$ has a root $\gamma$ in $\mathbb{F}_q$. Therefore, $F_p(\gamma)\subseteq \mathbb{F}_q$, i.e., $r$ is even.
\end{proof}

\begin{thm}\label{thm:n-subset}
Let $q=p^{2e}$. Denote by $M(n,0, 0)$ the number of $n$-subsets of $\mathbb{F}_q$ whose elements are the  solutions of the equations
\be\label{def:equations1}
\Bigg\{
\begin{array}{lll}
x_1^2+\cdots+x_n^2&=&0,\\
x_1+\cdots+x_n&=&0.
\end{array}
\ee

$i)$\ For $p\nmid n$, 
$$M(n, 0, 0)=\frac{1}{q^2}
\left(\begin{array}{c}
q\\n
\end{array}\right)
+\frac{q-1}{2\sqrt{q^3}}(\alpha(n)-(-1)^n\beta(n)).$$

$ii)$\ For $p|n$, $$
M(n, 0, 0)=\frac{1}{q^2}
\left(\begin{array}{c}
q\\n
\end{array}\right)
+\frac{q-1}{q^2}\left(\begin{array}{c}q/p\\n/p\end{array}\right)+\frac{q-1}{2q}(\alpha(n)+(-1)^n\beta(n)).$$
\end{thm}

\begin{proof}
Let $X$ be the set of all solutions of the equations \eqref{def:equations1}. Then $X$ is symmetric, so we have 
$$n!M(n, 0,0)=\sum\limits\limits_{\tau\in C_n}(-1)^{n-l(\tau)}C(\tau)|X_\tau|.$$
For a type $(c_1, c_2, \cdots, c_n)$  permutation $\tau$, we have $\sum\limits\limits_{i=1}^nic_i=n$ and $l(\tau)=\sum\limits\limits_{i=1}^nc_i$. Denote by $r$ the number of the cycles of $\tau$ such that the length of it is divisible by $p$, and denote by $s$ the number of the cycles of $\tau$ such that the length of it is not divisible by $p$. Note that $r+s=l(\tau)$ and $\sum\limits_{p\nmid i}ic_i\equiv n\pmod p$.

$i)$ Since $p\nmid n$, we have $s\geq 1$. If $s=1$, then $|X_\tau|=q^{l(\tau)-1}$. If $s\geq 2$, then by $i)$ of Lemma $\ref{lem:equations}$, we have 
$$|X_\tau|=\left\{\begin{array}{ll}
q^rq^{s-2}&{\rm if}\ s\ {\rm is\ even},\\
q^r(q^{s-2}+q^{(s-3)/2}(q-1)\chi((-1)^{(s-1)/2}\prod\limits_{p\nmid i}i^{c_i}\sum\limits_{p\nmid i}ic_i)&{\rm if}\ s\ {\rm is\ odd},\end{array}\right.$$
Since $q=p^{2e}$, by Lemma $\ref{lem:CHAR}$ we have 
$$|X_\tau|=\left\{\begin{array}{ll}
q^{l(\tau)-2}&{\rm if}\ s\ {\rm is\ even},\\
q^{l(\tau)-2}+(q-1)q^{-\frac{3}{2}}\prod\limits_{p|i}q^{c_i}\prod\limits_{p\nmid i}\sqrt{q}^{c_i}&{\rm if}\ s\ {\rm is\ odd},\end{array}\right.$$

Therefore, we have

$$
\begin{array}{rl}
M(n,0, 0)=&\frac{1}{n!}\sum\limits\limits_{\sum\limits\limits_{i}ic_i=n, s\geq 2}(-1)^{n-l(\tau)}C(\tau)q^{l(\tau)-2}+\frac{1}{n!}\sum\limits\limits_{\sum\limits\limits_{i}ic_i=n, s=1}(-1)^{n-l(\tau)}C(\tau)q^{l(\tau)-1}\\
&+\frac{1}{n!}(q-1)q^{-\frac{3}{2}}\sum\limits\limits_{\sum\limits\limits_{i}ic_i=n, s>2\ {\rm is\ odd}}(-1)^{n-l(\tau)}C(\tau)\prod\limits_{p|i}q^{c_i}\prod\limits_{p\nmid i}\sqrt{q}^{c_i}\\
=&\frac{1}{n!}\sum\limits\limits_{i=1}^n(-1)^{n-i}c(n,i)q^{i-2}\\
&+\frac{1}{n!}(q-1)q^{-\frac{3}{2}}\sum\limits\limits_{\sum\limits\limits_{i}ic_i=n, s\ {\rm is\ odd}}(-1)^{n-l(\tau)}C(\tau)\prod\limits_{p|i}q^{c_i}\prod\limits_{p\nmid i}\sqrt{q}^{c_i}\\
=&\frac{1}{q^2}
\left(\begin{array}{c}
q\\n
\end{array}\right)
+\frac{1}{n!}(-1)^n(q-1)q^{-\frac{3}{2}}S_-(n)\\
=&\frac{1}{q^2}
\left(\begin{array}{c}
q\\n
\end{array}\right)
+\frac{q-1}{2\sqrt{q^3}}(\alpha(n)-(-1)^n\beta(n)).
\end{array}
$$
  
 $ii)$ Since $p|n$, we have $s\neq 1$. If $s=0$, then $|X_\tau|=q^{l(\tau)}$. Denote by $CP_n$ the conjugacy classes in $C_n$ whose every cycle length is divisible by $p$, and denote by $p(n,i)$  the number of permutations in $S_n$ of $i$ cycles with the length of its each cycle divisible by $p$. If $s>0$, then by $iii)$ of Lemma $\ref{lem:equations}$ and $\ref{lem:CHAR}$ we have 

$
\begin{array}{rl}
M(n,0, 0)=&\frac{1}{n!}\sum\limits\limits_{\tau\notin CP_n}(-1)^{n-l(\tau)}C(\tau)|X_\tau|+\frac{1}{n!}\sum\limits\limits_{\tau\in CP_n}(-1)^{n-l(\tau)}C(\tau)|X_\tau|\\
 =&\frac{1}{n!}\sum\limits\limits_{\tau\notin CP_n}(-1)^{n-l(\tau)}C(\tau)|X_\tau|+\frac{1}{n!}\sum\limits\limits_{i=1}^n(-1)^{n-i}p(n,i)q^i\\
 =&\frac{1}{n!}\sum\limits\limits_{i=1}^n(-1)^{n-i}(c(n,i)-p(n,i))q^{i-2}+\frac{1}{n!}\sum\limits\limits_{i=1}^n(-1)^{n-i}p(n,i)q^i\\
&+\frac{1}{n!}(q-1)q^{-1}\sum\limits\limits_{s>0\ {\rm is\ even}}(-1)^{n-l(\tau)}C(\tau)\prod\limits_{p|i}q^{c_i}\prod\limits_{p\nmid i}\sqrt{q}^{c_i}\\
=&\frac{1}{n!}\sum\limits_{i=1}^n(-1)^{n-i}c(n,i)q^{i-2}+\frac{1}{n!}\frac{q-1}{q^2}\sum\limits\limits_{i=1}^n(-1)^{n-i}p(n,i)q^{i}\\
&+\frac{1}{n!}(q-1)q^{-1}\sum\limits\limits_{s\ {\rm is\ even}}(-1)^{n-l(\tau)}C(\tau)\prod\limits_{p|i}q^{c_i}\prod\limits_{p\nmid i}\sqrt{q}^{c_i}
\end{array}.
$

Recall that $\sum\limits_{i=1}^n(-1)^{n-i}p(n,i)q^{i}=(-1)^{n+\frac{n}{p}}n!\left(\begin{array}{c} q/p\\n/p\end{array}\right)$ (See Lemma 3.1 in $\cite{LW2}$) and $p$ is an odd prime number. Therefore,  
$$\begin{array}{rl}M(n,0, 0)=&\frac{1}{q^2}
\left(\begin{array}{c}q\\n\end{array}\right)+\frac{q-1}{q^2}\left(\begin{array}{c}q/p\\n/p\end{array}\right)+\frac{1}{n!}\frac{q-1}{q}S_+(n)\\
=&\frac{1}{q^2}
\left(\begin{array}{c}
q\\n
\end{array}\right)
+\frac{q-1}{q^2}\left(\begin{array}{c}q/p\\n/p\end{array}\right)+\frac{q-1}{2q}(\alpha(n)+(-1)^n\beta(n)).
\end{array}
$$
\end{proof}

\begin{cor}\label{cor:cor1}
Let $q=p^{2e}$. Denote by $M^{\prime}(n,0,0)$ the number of $n$-subsets of $\mathbb{F}_q$ whose elements are the common solutions of the equations

\be\label{def:equations2}
\left\{\begin{array}{ccc}
\sum\limits_{1\leq i<j\leq n}x_ix_j&=&0,\\
\sum\limits_{1\leq i\leq n}x_i&=&0.
\end{array}\right.
\ee

Then $M^{\prime}(n,0,0)=M(n,0,0)$
\end{cor}

\begin{proof}
 This result follows from the fact that the equations \eqref{def:equations2} are equivalent to the equations \eqref{def:equations1}.
\end{proof}

Let $X$ be the set of all solutions of the equations \eqref{def:equations2} in $\mathbb{F}_q^n$. Put $$X_1=\{(x_1, \cdots, x_n)\in X|x_i\neq x_j,\ \forall 1\leq i\neq j\leq n-1\}.$$ Now the symmetric group $S_{n-1}$ acts on $\mathbb{F}_q^n$ by permuting the first $n-1$ coordinates. The similar arguments of the $|\overline{X}|$ show that 
 $$|X_1|=\sum\limits\limits_{\tau\in C_{n-1}}(-1)^{n-1-l(\tau)}C(\tau)|X_\tau|,$$
where $C(\tau)$ is the number of permutations conjugate to $\tau\in S_{n-1}$ and $l(\tau)$ is the number of cycles including the trivial cycle. 
Denote by $M_1(n,0,0)$ the number of $n$-subsets $\{x_1, x_2, \cdots, x_n|x_i\neq x_j,\ \forall 1\leq i\neq j\leq n-1\}$ of $\mathbb{F}_q$ whose elements are the common solutions of the equations \eqref{def:equations2}. The similar proof of Theorem $\ref{thm:n-subset}$ shows  the following theorem:

\begin{thm}\label{thm:n-1-subset}
Let $q=p^{2e}$. Then

$i)$\ For $p\nmid n$ , 
$$M_1(n, 0, 0)=\frac{1}{q}
\left(\begin{array}{c}
q\\n-1
\end{array}\right)
+\frac{q-1}{2q}(\alpha(n-1)+(-1)^{n-1}\beta(n-1)).$$

$ii)$\ For $p|n$, 
$$M_1(n, 0, 0)=\frac{1}{q}
\left(\begin{array}{c}
q\\n-1
\end{array}\right)
+\frac{q-1}{2\sqrt{q}}(\alpha(n-1)-(-1)^{n-1}\beta(n-1)),$$
where $$\alpha(n-1)=\sum\limits_{i+pj=n-1, 0\leq i\leq \sqrt{q}}\left(\begin{array}{c}\sqrt{q}\\i\end{array}\right)\left(\begin{array}{c}\frac{q-\sqrt{q}}{p}\\ j\end{array}\right)$$ and $$\beta(n-1)=\sum\limits_{i+pj=n-1, i\geq 0}(-1)^j\left(\begin{array}{c}\sqrt{q}-1+i\\ \sqrt{q}-1\end{array}\right)\left(\begin{array}{c}\frac{q+\sqrt{q}}{p}\\ j\end{array}\right).$$
\end{thm}

\begin{thm}\label{thm:a1=a2=0}
Let $q=p^{2e}$ and  $k\leq n-1$. 

(i)\ If $p\nmid n$, then 

$$\begin{array}{ll}N_k(x^n,n-3)=&q^{n-k-2}\left(\begin{array}{c}
q\\k
\end{array}\right)\sum\limits\limits_{i=0}^{n-k}(-1)^i\left(\begin{array}{c}
q-k\\i
\end{array}\right)q^{-i}\\
&+(-1)^{n-k-1}\left(\begin{array}{c}
n-1\\k
\end{array}\right)\frac{q-1}{2q}D(n-1)+(-1)^{n-k}\left(\begin{array}{c}
n\\k
\end{array}\right)\frac{q-1}{2\sqrt{q^3}}D(n),\end{array}$$
where $D(n-1)=\alpha(n-1)+(-1)^{n-1}\beta(n-1)$ and $D(n)=\alpha(n)-(-1)^{n}\beta(n)$.

(ii)\ If $p|n$, then $$\begin{array}{ll}N_k(x^n,n-3)=&q^{n-k-2}\left(\begin{array}{c}
q\\k
\end{array}\right)\sum\limits\limits_{i=0}^{n-k}(-1)^i\left(\begin{array}{c}
q-k\\i
\end{array}\right)q^{-i}+(-1)^{n-k}\frac{q-1}{q^2}\left(\begin{array}{c}n\\k\end{array}\right)\left(\begin{array}{c}q/p\\n/p\end{array}\right)\\
&+(-1)^{n-k-1}\left(\begin{array}{c}
n-1\\k
\end{array}\right)\frac{q-1}{2\sqrt{q}}P(n-1)+(-1)^{n-k}\left(\begin{array}{c}
n\\k
\end{array}\right)\frac{q-1}{2q}P(n),\end{array}$$ 
where $P(n-1)=\alpha(n-1)-(-1)^{n-1}\beta(n-1)$ and $P(n)=\alpha(n)+(-1)^{n}\beta(n)$.
\end{thm}

\begin{proof}
Note that $$\left(\begin{array}{c}
n-1\\k
\end{array}\right)\left(\begin{array}{c}
q\\n-1
\end{array}\right)=\left(\begin{array}{c}
q\\k
\end{array}\right)\left(\begin{array}{c}
q-k\\n-k-1
\end{array}\right)\ {\rm and}\ \left(\begin{array}{c}
n\\k
\end{array}\right)\left(\begin{array}{c}
q\\n
\end{array}\right)=\left(\begin{array}{c}
q\\k
\end{array}\right)\left(\begin{array}{c}
q-k\\n-k
\end{array}\right).$$

This result follows from the similar proof of Theorem $\ref{thm:a1=0}.$
\end{proof}

\begin{rem}\label{rem:remark3}

$1)$ In the special case $k=n$, $N_k(x^n, n-3)=M(n,0,0).$

$2)$ For $n\geq q$,  the formulae of $N_k(x^n, n-3)$ can be obtained from the similar arguments of Remark $\ref{rem:remark1}$.

$i)$ If $n>q+1$, then $N_k(x^n, n-3)=q^{n-q-2}{q \choose k}(q-1)^{q-k}.$

$ii)$ If $n=q+1$, then$$
N_q(x^n, n-3)=\left\{\begin{array}{ll}1& k=q,\\ 0& k=q-1,\\ \frac{q-1}{q} {q \choose k}((q-1)^{q-k-1}+(-1)^{q-k})& k\leq q-2.\end{array}\right.$$

$ii)$ If $n=q$, then
$$
N_q(x^n, n-3)=\left\{\begin{array}{ll}1& k=q,\\ 0& k=q-1,\ q-2\\ \frac{q-1}{q} {q \choose k}(\frac{(q-1)^{q-k-1}}{q}+(-1)^{q-k-1}(q-k)+(-1)^{q-k}\frac{q+1}{q} )& k\leq q-3.\end{array}\right.$$
\end{rem}

\section{The spectrum of some jumped Wenger graphs}

In $\cite{Wenger}$, Wenger introduced a family of $p$-regular bipartite graphs and then Lazebnik and Ustimenko arrived at a family of bipartite graphs using a construction based on a certain Lie algebra for a prime power $q$ in $\cite{LU93}$. Later, Lazebnik and Viglione gave an equivalent representation of these graphs in $\cite{LV02}$. Then another useful representation of these graphs was given in $\cite{Viglione02}$, on which we concentrate in this section. All graph theory notions can be found in $\cite{Bollobas}$.

Let $m\geq 1$ be a positive integer and $g_k(x,y)\in \mathbb{F}_q[x,y]$ for $2\leq k\leq m+1$. Let $\mathfrak{P}=\mathbb{F}_q^{m+1}$ and $\mathfrak{L}=\mathbb{F}_q^{m+1}$ be two copies of the $(m+1)-$dimensional vector space over $\mathbb{F}_q$, which are called the point set and the line set respectively. If $a\in \mathbb{F}_q^{m+1}$, then we write $(a)\in\mathfrak{P}$ and $[a]\in\mathfrak{L}$. Denote $\mathfrak{G}=G_q(g_2,\cdots, g_{m+1})=(V,E)$ by the bipartite graph with vertex set $V=\mathfrak{B}\bigcup\mathfrak{L}$ and the edge set $E$ is defined as follows: there is an edge from a point $P=(p_1,p_2,\cdots,p_{m+1})\in\mathfrak{B}$ to a line $L=[l_1,l_2,\cdots,l_{m+1}]\in\mathfrak{L}$, denoted by $P~L$, if the following $m$ equalities hold:

\be\label{def:edge}
\Bigg\{
\begin{array}{lll}
l_2+p_2&=&g_2(p_1, l_1),\\
l_3+p_3&=&g_3(p_1, l_1),\\
&\vdots&\\
l_{m+1}+p_{m+1}&=&g_{m+1}(p_1, l_1).
\end{array}
\ee

If $g_k(x,y)=x^{k-1}y,\ k=2,\cdots, m+1$, then the graph is just the original Wenger graph denoted by  $W_m(q)$ in $\cite{CL14}$. The spectrum, the diameter and the automorphism group of $W_m(q)$ were studied in $\cite{CL14},\ \cite{LU93},\ \cite{LU95}$ and $\cite{Viglione08}$. In $\cite{CLWWW}$,  a new class of bipartite graphs called linearized Wenger graphs was introduced. These graphs were defined by $\eqref{def:edge}$ together with $g_k(x,y)=x^{p^{k-2}},\ k=2,\cdots, m+1$, which denoted by $L_m(q)$. When $m\geq e$, the spectrum of $L_m(q)$ was explicitly determined using results on linearized polynomials over finite fields. The diameter and girth of $L_m(q)$ also were obtained. Furthermore,  the spectrum of a general class of graphs, which defined  by $g_k(x,y)=f_k(x)y$ and the mapping $\sigma:\ \mathbb{F}_q\rightarrow \mathbb{F}_q^{m+1};\ u\rightarrow (1, f_2(u), \cdots, f_{m+1}(u))$ is injective, was studied. The eigenvalues of such a graph were determined  and their multiplicities were reduced to counting certain polynomials with a given number of roots over $\mathbb{F}_q$. That is, for all prime power $q$ and positive integer $m$, the eigenvalues of $\mathfrak{G}$, counted with multiplicities, are 
$$\pm\sqrt{qN_{F_\omega}},\ \ \ (\omega_1, \omega_2, \cdots, \omega_{m+1})\in \mathbb{F}_q^{m+1},$$ where $F_\omega(u)=\omega_1+\omega_2f_2(u)+\cdots+\omega_{m+1}f_{m+1}(u)$ and $N_{F_\omega}=|\{u\in\mathbb{F}_q:\ F_\omega(u)=0\}|$. For $0\leq i\leq q$, the multiplicity of $\pm\sqrt{qi}$ is $$n_i=|\{\omega\in \mathbb{F}_q^{m+1}:\ N_{F_\omega}=i\}|.$$

In this section, we use our previous results to get the spectrum of a general class of graphs defined by $\eqref{def:edge}$  together with polynomials $g_k(x,y)=f_k(x)y,\in \mathbb{F}_q[x,y]$ , where $f_k(x)=x^{k-1},\ 2\leq k\leq m, f_{m+1}(x)=x^{m+1}$ and $f_k(x)=x^{k-1},\ 2\leq k\leq m, f_{m+1}(x)=x^{m+2}$. These graphs are denoted by $JW_m^1(q)$ and $JW_m^2(q)$ respectively.

\begin{thm}\label{thm:JW1}
For all prime power $q$ and $1\leq m+1\leq q-1$, the distinct eigenvalues of $JW_m^1(q)$ are 
$$\pm q,\ \pm\sqrt{(m+1)q},\ \pm\sqrt{mq},\ \cdots,\ \ \pm\sqrt{2q},\ \pm\sqrt{q},\ 0.$$
The multiplicity of the eigenvalue $\pm\sqrt{iq}$ is 
$$(q-1)\left(\begin{array}{c}q\\ i\end{array}\right)\sum\limits_{d=i}^{m-1}\sum\limits_{k=0}^{d-i}(-1)^k\left(\begin{array}{c}q-i\\ k\end{array}\right)q^{d-i-k}+(q-1)N_i(x^{m+2},m-1),\ 0\leq i\leq m-1$$
and 
$$(q-1)N_i(x^{m+1},m-1),\ i=m,\ m+1.$$
\end{thm}

\begin{proof}
Let $(\omega_1, \omega_2, \cdots, \omega_{m+1})\in \mathbb{F}_q^{m+1}$ and $f(X)=\omega_1+\omega_2X+\cdots\omega_mX^{m-1}+\omega_{m+1}X^{m+1}$.

In the case of $f=0$, $|\{u\in \mathbb{F}_q|f(u)=0\}|=q$. Thus, $JW_m^1(q)$ has $\pm q$ as its eigenvalues.

For any $0\leq i\leq m-1$, there exists a polynomial $f$ over $\mathbb{F}_q$ of degree at most $m+1\leq q-1$, which has exactly $i$ distinct roots in $\mathbb{F}_q$. 

For $i=m, m+1$,  it is easy to compute $N_1(m+1, i)>0$ by Theorem $\ref{thm:a1=0}$.  Then there exists a polynomial of degree $m+1$ which has exactly $i$ distinct roots in $\mathbb{F}_q$. Thus by Theorem 2.2 in $\cite{CLWWW}$, $JW_m^1(q)$ has $\pm\sqrt{iq},\ 0\leq i\leq m+1$ as its eigenvalues, and by Theorem  $\ref{thm:n-l=1}$ and Theorem $\ref{thm:a1=0}$, we obtain that the multiplicity of the eigenvalue $\pm\sqrt{iq}$ is 
$$(q-1)\left(\begin{array}{c}q\\ i\end{array}\right)\sum\limits_{d=i}^{m-1}\sum\limits_{k=0}^{d-i}(-1)^k\left(\begin{array}{c}q-i\\ k\end{array}\right)q^{d-i-k}+(q-1)N_i(x^{m+1},m-1),\ 0\leq i\leq m-1$$
and 
$$(q-1)N_i(x^{m+1},m-1),\ i=m,\ m+1.$$
\end{proof}

Similarly, we have the following theorem about the spectrum of $JW_m^2(q)$ by Theorem $\ref{thm:n-subset}$.

\begin{thm}\label{thm:JW2}
Suppose that $p$ is an odd prime number and $q=p^{2e}$.  If $1\leq m+2\leq q-1$, then we have

$1)$
$$\pm q,\ \pm\sqrt{(m-1)q},\ \pm\sqrt{(m-2)q},\ \cdots,\ \ \pm\sqrt{2q},\ \pm\sqrt{q},\ 0$$ are the distinct eigenvalues of $JW_m^2(q)$.
The multiplicity of the eigenvalue $\pm\sqrt{iq}$ is 
$$(q-1)\left(\begin{array}{c}q\\ i\end{array}\right)\sum\limits_{d=i}^{m-1}\sum\limits_{k=0}^{d-i}(-1)^k\left(\begin{array}{c}q-i\\ k\end{array}\right)q^{d-i-k}+(q-1)N_i(x^{m+2},m-1),\ 0\leq i\leq m-1.$$

$2)$ If $m\leq i\leq m+2$, then $\pm\sqrt{iq}$ are the distinct eigenvalues of $JW_m^2(q)$ if and only if $N_i(x^{m+2},m-1)>0$. Furthermore, the multiplicity of the eigenvalue $\pm\sqrt{iq}$ is 
$$(q-1)N_i(x^{m+2},m-1).$$

\end{thm}


\end{document}